\documentclass{article}

\usepackage[latin2]{inputenc} 
\usepackage {latexsym} \usepackage {amsmath} \usepackage {amsfonts} 
\usepackage {amssymb}
\usepackage {times}
\usepackage {graphicx}
\usepackage {paralist}

\let\phi\varphi

\def\Z{{\mathbb Z}}
\let\zet\Z

\let\en\N

\def\lcm{\mathop{\rm lcm}\nolimits}

\let\myarrow\overrightarrow 
\def\orC{\myarrow C}
\def\orD{\myarrow D}

\def\orK{\myarrow K}

\newcommand{\FF}[1][]{\mathrel{\xrightarrow{\ifx @#1@ FF\else FF_{#1} \fi}}}
\newcommand{\nFF}[1][]{\thickspace\not\negthickspace\xrightarrow{\ifx @#1@ FF\else FF_{#1} \fi}}

\newtheorem{theorem}{Theorem}[section]
\newtheorem{lemma}[theorem]{Lemma}
\newtheorem{corollary}[theorem]{Corollary}

\newtheorem{remark}[theorem]{Remark}

\newenvironment{proof}{\par\medskip\noindent{\bf Proof: }}
  {\unskip\hfill$\Box$\par\medskip}
\newenvironment{proofof}{\par\medskip\proofof}
  {\unskip\hfill$\Box$\par\medskip}
\def\proofof #1{\noindent{\bf Proof of #1:}}

\let\qed\relax
\title{Flow-continuous mappings -- influence of the group\thanks{Partially supported by grant LL1201 ERC CZ of the Czech Ministry of Education, Youth and Sports.}} 
\author{
Jaroslav Nešetřil 
 \qquad
Robert Šámal \thanks{Partially supported by grant GA \v{C}R P201/10/P337.}
\\
\small
 Computer Science Institute, Charles University \\ 
\small
 email: {\tt \{nesetril,samal\}@iuuk.mff.cuni.cz} 
}

\begin{document}
\date{}
\maketitle

\begin{abstract}
Many questions at the core of graph theory can be formulated as 
questions about certain group-valued flows: examples are the cycle double cover 
conjecture, Berge-Fulkerson conjecture, and Tutte's 3-flow, 4-flow, and 5-flow 
conjectures. As an approach to these problems Jaeger and DeVos, Nešetřil, 
and Raspaud define a notion of graph morphisms continuous with respect to 
group-valued flows. We discuss the influence of the group on these 
maps. In particular, we prove that the number of flow-continuous mappings 
between two graphs does not depend on the group, but only on the 
largest order of an element of the group (i.e., on the exponent of the group). 
Further, there is a nice algebraic structure describing for which groups a mapping is flow-continuous. 

On the combinatorial side, we show that for cubic graphs the only relevant 
groups are $\Z_2$, $\Z_3$, and $\Z$. 
\end{abstract}

\section{Introduction} 
\label{sec:intro}

Throughout this paper $G$ and~$H$ will be digraphs (finite multidigraphs
with loops and parallel edges allowed), 
$f: E(G) \to E(H)$ a mapping, and $M$, $N$ abelian groups.

Recall that a mapping $\phi: E(G) \to M$ is a \emph{flow}
($M$-flow when we want to emphasize~$M$) when it satisfies 
the Kirchhoff's law at every vertex, that is, for every $v \in V(G)$ we have
$$
   \sum_{e \in E(G): \hbox{$e$ leaves $v$}} \phi(e) = \sum_{e \in E(G): \hbox{$e$ enters $v$}} \phi(e) \,.  
$$
The theory of flows on (di)graphs is a very rich one, but also full 
of longstanding conjectures (cycle double cover, Berge-Fulkerson conjecture, Tutte's 3-flow, 4-flow, and 5-flow conjectures, etc.), 
see \cite{seymour-handbook}, \cite{Diestel}, or \cite{zhang-book} for a detailed treatment of this area. 

In this paper we are going to study a notion introduced by 
Jaeger \cite{Jaeger} and by DeVos, Nešetřil, and Raspaud \cite{DNR}
as an approach to these problems.  

We say that a mapping $f:E(G) \to E(H)$ is $M$-flow continuous, if ``the preimage of every $M$-flow 
is an $M$-flow''. More precisely, for every $M$-flow $\phi$ on~$H$, the composition $\phi f$ 
(applying first $f$ then $\phi$) is an $M$-flow on~$G$. For short, we will call $M$-flow-continuous 
mappings just $FF_M$; 
in the important case $M=\Z_n$ we use the typographically nicer~$FF_n$ instead of~$FF_{\Z_n}$. 
We will write $G \FF[M] H$ to denote that there exists some $FF_M$ mapping from~$G$ to~$H$. 

The main reason for introducing this notion is Jaeger's conjecture \cite{Jaeger} that 
every bridgeless cubic graph~$G$ has a $\Z_2$-flow-continuous mapping to the Petersen graph. 
If true, this conjecture would imply the cycle double cover conjecture, and many others. 
In this paper we will study the notion of $M$-flow-continuous mappings per se, with the 
aim of making clear what the role of the group~$M$ is; this question has not been addressed in previous 
treatments. For $M=\Z_2$ we do not need to consider the orientation of edges, thus this part of the 
theory is relevant for undirected graphs. As our emphasis is on general abelian groups, we will 
mostly deal with digraphs. 

In some of our proofs we will use an alternative characterisation of $FF$-mappings, 
to state it we need to briefly introduce two notions. 
Given $\tau: E(G) \to M$ and $f: E(G) \to E(H)$, we 
denote by $\tau_f$ the \emph{algebraic image of~$\tau$}, i.e., 
the mapping $\tau_f: E(H) \to M$ defined by 
$$
  \tau_f(e) = \sum_{e' \in E(G); f(e')=e} \tau(e') \,. 
$$ 

A mapping $t: E(G) \to M$ is called an \emph{$M$-tension} if for every 
circuit~$C$ the sum of~$t$ over all clockwise edges is the same as the sum 
over all counterclockwise edges. It is not hard to see that $M$-tensions 
in a plane digraph~$G$ correspond to $M$-flows in the dual~$G^*$. More relevant for us 
is that for every digraph the vector spaces of all $M$-flows and of all $M$-tensions 
are orthogonal complements. (For this we need $M$~to be a ring. As we will be restricted 
on finitely generated abelian groups, i.e., on groups in the form~\eqref{eq:M}, this 
will not limit our use of the following lemma.) 
This allowed DeVos, Nešetřil, and Raspaud~\cite[Theorem~3.1]{DNR} to prove the following useful result. 

\begin{lemma}\label{altdef}
  Let $f: E(G) \to E(H)$ be a mapping, let $M$ be a \emph{ring}. 
  Mapping $f$ is $FF_M$ if and only if for every $M$-tension~$\tau$ on~$G$, 
  its algebraic image~$\tau_f$ is an~$M$-tension on~$H$. 

  Moreover, it is sufficient to verify the condition for all tensions 
  that are nonzero only on a neighborhood of a single vertex. 
\end{lemma}

As an easy corollary of this lemma, we observe that $FF_2$-mappings between cubic 
bridgeless graphs map a 3-edge-cut to a 3-edge-cut. In particular, if the target 
graph is cyclically 4-edge-connected, then the image of an elementary cut (all edges 
around a vertex) is an elementary cut. 

\section{Influence of the group}
\label{sec:group}

In this section we study how the notion of $M$-flow-continuous 
mapping depends on the group~$M$. 
Although the existence of $M$-flow-continuous mappings seems
to be strongly dependent on the choice of~$M$ we prove here 
(in Theorem~\ref{numberofFF}) that this dependence
relates only to the largest order of an element of~$M$. 

As we are interested only in finite digraphs, we can restrict our
attention to finitely generated groups---clearly $f$ is
$M$-flow-continuous if and only if it is $N$-flow-continuous for
every finitely generated subgroup~$N$ of~$M$.
Hence, there are integers
$\alpha$, $k$, $\beta_i$, $n_i$ ($i = 1, \dots, k$)
so that
\begin{equation}\label{eq:M}
  M \simeq 
     \zet^\alpha \times \prod_{i=1}^k \zet_{n_i}^{\beta_i} \,.
\end{equation}
Note that each such group has a canonical ring structure, thus we will 
be able to use Lemma~\ref{altdef}. 

For a group $M$ in the form~\eqref{eq:M}, let $n(M) = \infty$ if
$\alpha > 0$, otherwise let $n(M)$ be the least common multiple of 
$\{n_1, \dots, n_k\}$. When $n(M)$ is finite, it is called the 
\emph{exponent} of the group~$M$. 
An alternative definition is that 
$n(M)$ is the largest order of an element of~$M$ (here 
order of $a \in M$ is the smallest $n > 0$ such that 
$n\cdot a = a+a+ \cdots + a = 0$). 

As a first step to a complete characterization we consider a  
specialized question: given a $FF_M$~mapping, when can we conclude
that it is $FF_N$ as well?

\begin {lemma}   \label{MN}
\begin {enumerate}
  \item If $f$ is $FF_\zet$ then it is $FF_M$ for any abelian group~$M$.
  \item Let $M$ be a subgroup of abelian group~$N$. If $f$ is $FF_N$ then it
    is $FF_M$.
\end {enumerate}
\end {lemma}

\begin {proof}
1. This appears as Theorem~4.4 in~\cite{DNR}.

2. Let $\phi$ be an $M$-flow on~$H$. As $M \le N$, we may
regard $\phi$ as an $N$-flow, hence $\phi f$ is an $N$-flow
on~$G$. As it attains only values in the range of $\phi$, hence
in~$M$, it is an $M$-flow, too.
\qed
\end {proof}

\begin {lemma}   \label{M12}
Let $M_1$, $M_2$ be two abelian groups.
Mapping $f$ is $FF_{M_1}$ and $FF_{M_2}$ if and only if
it is $FF_{M_1 \times M_2}$.
\end {lemma}

\begin {proof}
As $M_1$, $M_2$ are isomorphic to subgroups of $M_1 \times M_2$, one implication
follows from the second part of Lemma~\ref{MN}. For the other implication
let $\phi$ be an $(M_1 \times M_2)$-flow on~$H$. Write 
$\phi = (\phi_1, \phi_2)$, where $\phi_i$ is an $M_i$-flow on~$H$.
By assumption, $\phi_i f$ is an $M_i$ flow on~$G$, consequently
$\phi f = (\phi_1 f, \phi_2 f)$ is a flow too.
\qed
\end {proof}

The following (somewhat surprising) lemma shows that we can restrict
our attention to cyclic groups only.

\begin {lemma}   \label{onlycyclic}
\begin {enumerate}
  \item If $n(M) = \infty$ then $f$ is $FF_M$ if and only if it is $FF_\zet$.
  \item Otherwise $f$ is $FF_M$ if and only if it is $FF_{n(M)}$.
\end {enumerate}
\end {lemma}

\begin {proof}
In the first part, each implication follows from one part of Lemma~\ref{MN}. 
In the second part: If $f$ is $FF_M$, we use the fact that 
$\zet_{n(M)}$ is isomorphic to a subgroup of~$M$, 
thus the second part of Lemma~\ref{MN} implies $f$ is~$FF_{n(M)}$. 
For the other implication, suppose that $f$ is $FF_{n(M)}$. 
Note that whenever $\Z_{n_i}$ occurs in the expression~\eqref{eq:M}
for~$M$, then $\Z_{n_i}$ is a subgroup of~$\Z_{n(M)}$. 
Consequently (Lemma~\ref{MN}, second part) $f$ is $FF_{n_i}$. 
Repeated application of~Lemma~\ref{M12} implies $f$ is $FF_M$ as well. 
\qed
\end {proof}

By a theorem of Tutte \cite{tutte1}, for a finite abelian group~$M$, 
the number of nowhere-zero $M$-flows on a given (di)graph only depends on the order of~$M$
(see also~\cite[Chapter~6]{Diestel}). 
Before proceeding in the main direction of this section, let us
note a consequence of Lemma~\ref{onlycyclic}, which is an analogue of Tutte's result. 

\begin {theorem}   \label{numberofFF}
Given digraphs~$G$, $H$, the number of $FF_M$ mappings from~$G$ to~$H$ 
depends only on~$n(M)$.
\end {theorem}

Lemma~\ref{onlycyclic} suggests to define for two digraphs the set
$$
  FF(G,H) = \{ n \ge 1 \mid \hbox{there is $f: E(G) \to E(H)$ 
          such that $f$ is $FF_n$} \}
$$
and for a particular $f: E(G) \to E(H)$ 
$$
  FF(f,G,H) = \{ n \ge 1 \mid \hbox{$f$ is $FF_n$} \}  \,.
$$

We remark that most of these sets contain 1: $\zet_1$ is a trivial 
group, hence any mapping is $FF_1$. 
Therefore $1 \in FF(f,G,H)$ for every $f: E(G) \to E(H)$, 
while $1 \in FF(G,H)$ if and only if there exists a mapping $E(G) \to E(H)$.
This happens always, unless $E(H)$ is empty and $E(G)$ nonempty.

Although we are working with finite digraphs throughout the paper, 
in the following results we stress this---contrary to most
of the other results, these are not true for infinite digraphs.

\begin {lemma}   \label{FFf-fin}
Let $G$ be a finite digraph.
Either $FF(f,G,H)$ is finite or $FF(f,G,H)=\en$.
In the latter case $f$ is $FF_\zet$.
\end {lemma}

\begin {proof}
It is enough to prove that $f$ is $FF_\zet$ if it is $FF_n$
for infinitely many integers~$n$. To this end, take a $\zet$-flow 
$\phi$ on~$H$. As $\phi_n : e \mapsto \phi(e) \bmod n$ is a
$\zet_n$-flow, $\phi_n f = \phi f \bmod n$ is a $\zet_n$-flow
whenever $f$~is $FF_n$.
To show $\phi f$ is a $\zet$-flow consider a vertex~$v$ of~$G$ and 
let $s$ be the ``$\pm$-sum'' (in~$\zet$) around~$v$: 
$$
s = \sum_{\hbox{$e$ leaves $v$}}(\phi f)(e) - \sum_{\hbox{$e$ enters $v$}} (\phi f)(e) \,. 
$$  
As $s \bmod n = 0$ for infinitely many values of~$n$, we have $s=0$.
\qed
\end {proof}

Any $f$ induced by a local isomorphism is $FF_\Z$, 
thus providing an example where
$FF(f,G,H)$ is the whole of~$\en$. For finite sets the situation is more
interesting. By the next theorem, the sets $FF(f,G,H)$
are precisely the ideals in the divisibility lattice.

\begin {theorem}   \label{FFf-str}
Let $S$ be a finite subset of\/ $\en$. Then the following are
equivalent.
\begin {enumerate}
  \item There are $G$, $H$, $f$ such that $S = FF(f, G, H)$. 
  \item There is $n \in \en$ such that $S$ is the set of 
    all divisors of~$n$.
\end {enumerate}
\end {theorem}

\begin {proof}
First we show that 1 implies 2.
The set $S=FF(f,G,H)$ has the following properties
\begin {enumerate}[(i)]
  \item If $a \in S$ and $b | a$ then $b \in S$. 
    (We use the second part of Lemma~\ref{MN}: if $b$ divides $a$, then
    $\zet_b \le \zet_a$.)
  \item If $a, b \in S$ then the least common multiple of $a$, $b$
     is in~$S$. 
    (We use Lemma~\ref{MN} and Lemma~\ref{M12}: if 
    $l = \lcm(a,b)$ then $\zet_l \le \zet_a \times \zet_b$.)
\end {enumerate}

Let $n$ be the maximum of~$S$. By (i), all divisors of~$n$ are in~$S$.
If there is a $k \in S$ that does not divide~$n$ then $\lcm(k,n)$
is an element of~$S$ larger than~$n$, a contradiction.

For the other implication, let $\orD_n$ be a graph with two vertices 
and $n$~parallel edges in the same direction, let $L$ be a loop
(digraph with a single vertex and one edge). 
Let $f$ be the only mapping from~$\orD_n$ to~$L$. 
Then $FF(f,\orD_n,L) = S$: mapping~$f$ is~$FF_k$ if and only if for any $a \in \zet_k$ the constant mapping
$E(\orD_n) \mapsto a$ is a $\zet_k$-flow; this occurs
precisely when $k$~divides~$n$. 
\qed
\end {proof}

Let us turn to describing the sets $FF(G,H)$.

\begin {lemma}   \label{FF-fin}
Let $G$, $H$ be \emph{finite} digraphs.
Either $FF(G,H)$ is finite or $FF(G,H)=\en$.
In the latter case $G \FF[\zet] H$.
\end {lemma}

\begin {proof}
As in the proof of~Lemma~\ref{FFf-fin}, the only difficult
step is to show that if $G \FF[n] H$ for infinitely many 
values of~$n$, then $G \FF[\zet] H$. As $G$~and~$H$ are finite, 
there is only a finite number of possible mappings between
their edge sets. Hence, there is one of them, say $f: E(G) \to E(H)$, 
that is $FF_n$ for infinitely many values of~$n$. By
Lemma~\ref{FFf-fin} we have $f: G \FF[\zet] H$.
\qed
\end {proof}

When characterizing the sets $FF(G, H)$ we first remark that
an analogue of Lem\-ma~\ref{M12} does not hold: there is a 
$FF_M$~mapping from~$\orD_9$ to~$\orD_7$ for
$M=\zet_2$ (mapping that identifies three edges and is 1--1 
on the other ones) and for $M=\zet_3$ (e.g., a constant mapping), 
but not the same mapping for both, hence there
is no $FF_{\zet_2\times\zet_3}$~mapping. We will see that
the sets $FF(G,H)$ are precisely the down-sets in the divisibility poset.
First, we prove a lemma that will help us to construct pairs of digraphs
$G$, $H$ with a given $FF(G,H)$.
The integer cone of a set $\{s_1, \dots, s_t\} \subseteq \en$ is
the set $\{ \sum_{i=1}^t a_i s_i \mid a_i \in \zet, a_i \ge 0 \}$.

\begin {lemma}   \label{circuits}
Let $A$, $B$ be non-empty subsets of~$\en$, $n \in \en$, define
$G = \bigcup_{a \in A} \orD_a$, and~$H = \bigcup_{b \in B} \orD_b$.
Then there is a $FF_n$~mapping from~$G$ to~$H$
if and only if 
$$
  \hbox {$A$ is a subset of the integer cone of $B \cup \{ n\}$} \,.
$$
\end {lemma}

\begin {proof}
We use Lemma~\ref{altdef}.
Consider a tension $\tau_a$ taking the value 1 on $\orD_a$ and 0
elsewhere. The algebraic image of this tension is a tension, hence
it is (modulo $n$) a sum of several tensions on the digons 
$\orD_b$, implying $a$~is in integer cone of~$B \cup \{n\}$.
On the other hand if $a = \sum_i b_i + cn$ (with $b_i \in B$)
then we can map any $cn$ edges of~$\orD_a$ to one (arbitrary) edge of~$H$, 
and for each $i$ any (``unused'') $b_i$ edges bijectively to
$\orD_{b_i}$. After we have done this for each $a \in A$ we will have
constructed a $\zet_n$-flow-continuous mapping from~$G$ to~$H$.
\qed
\end {proof}

\begin {theorem}   \label{FF-str}
Let $S$ be a finite subset of~$\en$. Then the following are
equivalent.
\begin {enumerate}
  \item There are $G$, $H$  such that $S = FF(G, H)$. 
  \item There is a finite set $T \subset \en$ such that 
  $$
      S = \{ s \in \en ; (\exists t \in T) \quad s | t \} \,.
  $$
\end {enumerate}
\end {theorem}

\begin {proof}
If $S$ is empty, we take $T$ empty in part~2. In part~1, 
we just consider digraphs such that $E(H)$ is empty and
$E(G)$ is not. Next, we suppose $S$ is nonempty.

By the same reasoning as in the proof of Theorem~\ref{FFf-str}
we see that if $a \in FF(G,H)$ and $b | a$ then $b \in FF(G,H)$.
Hence, 1 implies 2, as we can take $T = S$ (or, to make
$T$ smaller, let $T$ consist of the maximal elements of~$S$ in the
divisibility relation). 

For the other implication let $p > 4 \max T$ be a prime, 
let $p' \in (1.25p, 1.5p)$ be an integer. Let $A = \{p,p'\}$ and
$$
  B = \{ p-t; t \in T\} \cup \{ p'-t; t \in T\} \,;
$$
note that every element of~$B$ is larger than $\tfrac34 p$.
As in Lemma~\ref{circuits} we define
$G = \bigcup_{a \in A} \orD_a$, $H = \bigcup_{b \in B} \orD_b$.
We claim that $FF(G,H) = S$. By Lemma~\ref{circuits} it is immediate
that $FF(G,H) \supseteq S$. For the other direction take
$n \in FF(G,H)$. By Lemma~\ref{circuits} again, we can express
$p$ and~$p'$ in the form
\begin{equation}\label{eq:formofp}
    \sum_{i=1}^k b_i + cn
\end{equation}
for integers $c, k \ge 0$, and $b_i \in B$.
\begin {itemize}
  \item If $k\ge 2$ then the sum in~(\ref{eq:formofp}) is at least $1.5p$;
     hence neither $p$ nor $p'$ can be expressed with $k\ge 2$.
  \item If $k=1$ then we distinguish two cases.
    \begin {itemize}[$\bullet$]
      \item $p = (p-t) + cn$, hence $n$ divides $t$ and thus $n\in S$.
      \item $p = (p'-t) + cn$, hence $p'-p \le t$. But
        $p'-p > 0.25p > t$, a contradiction.
    \end {itemize}
    Considering $p'$ we find that either $n \in S$ or
    $p' = (p-t) + cn$.
  \item Finally, consider $k=0$. If $p=cn$ then either $n=1$ (so $n \in S$)  
   or $n=p$. (We don't claim anything about $p'$.)
\end {itemize}
To summarize, if $n \in FF(G,H) \setminus S$ then necessarily $n=p$.
For $p'$ we have only two possible expressions: $p'=cn$ (for $k=0$) and
$p' = (p-t) + cn$ (for $k=1$). We easily check that both of them lead to 
a contradiction. The first one contradicts $1.25 p < p' < 1.5 p$. 
In the second expression $c=0$ implies $p'<p$
while $c\ge 1$ implies $p'\ge 2p-t \ge 1.75p$, again a contradiction.
\qed
\end {proof}

\begin{remark}
This paper concentrates on $FF$~mappings. We remark, however, that 
analogous proofs describe the role of the group for mappings where preimages of tensions are tensions, 
or preimages of tensions are flows (or preimages of flows are tensions). 
For a discussion of the relevance of these types of mappings the reader may 
consult the series \cite{ns-tt1,ns-tt2} by the authors and 
the second author's Ph.D. thesis~\cite{rsthesis}. 
\end{remark}

\section{Cubic graphs}
\label{sec:cubic}

In the previous section we studied how the group~$M$ 
influences the notion of~$FF_M$-mappings; it turned out there is an 
algebraic structure behind this. In this section we look at the 
combinatorially more relevant case of cubic graphs. (Degree of each 
vertex is~3, the orientation is arbitrary.) Indeed, many 
longstanding conjectures in the area have been reduced to the case of 
cubic graphs. There it turns out that we only need to consider 
three groups: $\Z_2$, $\Z_3$, and $\Z$.

\begin{theorem} \label{maincubic}
Let $n > 3$ be an integer, suppose $G$, $H$ are digraphs with maximum degree 
less than~$n$. 
Then $G \FF[n] H$ is equivalent with $G \FF[\Z] H$. 
\end{theorem}

\begin{proof}
One direction follows from Lemma~\ref{MN}. For the other one, 
consider a mapping~$f: E(G) \to E(H)$. We will show that if it is 
$FF_n$, it is $FF_\Z$ as well. 
Taking a $\Z$-flow~$\phi$ on~$H$, we will show that $\phi f$ 
is a $\Z$-flow on~$G$. We only need to test this on elementary flows 
(those taking only values $\pm 1$ around a circuit), as these 
form a basis for $\Z$-flows. So suppose $\phi$~is one of these; 
notice that it is both a $\Z$-flow and a $\Z_n$-flow. 
Thus, $\phi f$ is a $\Z_n$-flow on~$G$. 
Consider a vertex~$v \in V(G)$ of degree~$d<n$ and let $e_1$, $e_2$, \dots, $e_d$ be the edges 
incident with it; further, let $a_i = \phi(f(e_i))$. As $\phi f$ is a $\Z_n$-flow, 
we have that $s = \pm a_1 \pm a_2 \pm \dots \pm a_d \equiv 0 \pmod n$ (the signs are chosen 
based on orientation of the edges). Now $|s| \le d < n$, thus $s=0$. 
It follows that $\phi f$ also satisfies the Kirchhoff's law at~$v$ in~$\Z$, 
thus $\phi f$ is also a flow over~$\Z$. 
\end{proof}

\begin{corollary}
Let $G$, $H$ be digraphs of maximum degree~$3$, let $n>3$ be an integer.  
Then $G \FF[n] H$ is equivalent with $G \FF[\Z] H$. 
\end{corollary}

Together with Lemma~\ref{onlycyclic}, the above corollary implies 
that for subcubic digraphs we only need to consider $\Z_2$-, $\Z_3$-, 
and $\Z$-flow-continuous mappings. 

By Lemma~\ref{MN} $\Z$-flow continuous mapping is also $\Z_2$- and $\Z_3$-flow-continuous. 
In the following examples we show that existence of $FF_2$ and $FF_3$ mappings are independent, 
even for subcubic digraphs. Let $f$~be any bijection from $E(\orD_3)$ to $E(\orC_3)$. 
Mapping~$f$ is $FF_n$ only if $n$~is a multiple of~$3$, thus 
it is $FF_3$ but not $FF_2$ nor $FF_\Z$. 
On the other hand, consider an edge 3-coloring for~$\orK_4$ (a $K_4$ with an arbitrary orientation of edges) 
as a mapping $g: \orK_4 \to \orD_3$. This mapping is~$FF_2$ (as a 4-cycle in~$K_4$ is also a cut). 
However, $g$~is not~$FF_3$: consider a $\Z_3$-flow $\phi$ in~$\orD_3$ that equals~$1$ on all three edges of~$\orD_3$.
Clearly the composition $\phi f$ is not a $\Z_3$-flow on~$\orK_4$.

\section*{Acknowledgments} 

The author is grateful to the anonymous referees for their valuable comments.


\bibliographystyle{rs-amsplain}
{
  \bibliography{FF-group} 
}

\end{document}